\documentclass{amsart} 

\usepackage{amsmath,amssymb,latexsym, amscd}
\usepackage{exscale, cite, eps fig, graphics}

\usepackage[colorlinks]{hyperref}

\renewcommand{\geq}{\geqslant}
\renewcommand{\leq}{\leqslant}

\newtheorem{theorem}{Theorem}[section]
\newtheorem{lemma}[theorem]{Lemma}

\newtheorem*{main-theorem}{Main Theorem}
\newtheorem*{remark*}{Remark}
\numberwithin{equation}{section}

\title[Wave breaking in the Whitham equation]
{Wave breaking in the Whitham equation}

\author[Hur]{Vera~Mikyoung~Hur}
\address{Department of Mathematics, University of Illinois at Urbana-Champaign, Urbana, IL 61801 USA}
\email{verahur@math.uiuc.edu}

\date{\today}

\keywords{blow-up; wave breaking; Whitham equation; shallow water}
\subjclass[2010]{35A20, 35B44, 35S10, 35F25, 76B15}

\begin{document}

\maketitle

\begin{abstract}
We prove wave breaking --- bounded solutions with unbounded derivatives --- in the nonlinear nonlocal equation which combines the dispersion relation of water waves and a nonlinearity of the shallow water equations, provided that the slope of the initial datum is sufficiently negative, whereby we solve a Whitham's conjecture. We extend the result to equations of Korteweg-de Vries type for a range of fractional dispersion.
\end{abstract}

\section{Introduction}\label{sec:intro}

As Whitham~\cite[pp.~457]{Whitham} emphasized, ``the breaking phenomenon is one of the most intriguing long-standing problems of water wave theory." The {\em shallow water equations}:
\begin{equation}\label{E:shallow}
\begin{gathered}
\partial_t\eta+\partial_x((1+a\eta)u)=0,\\
\partial_tu+\partial_x\eta+au\partial_xu=0,
\end{gathered}
\end{equation}
approximate the physical problem, and they explain {\em wave breaking}. That is, the solution remains bounded but its slope becomes unbounded in finite time. Here $t\in\mathbb{R}$ is proportional to elapsed time, and $x\in\mathbb{R}$ is the spatial variable in the primary direction of wave propagation; $\eta=\eta(x,t)$ is the fluid surface displacement from the depth $=1$, $u=u(x,t)$ is the particle velocity at the rigid flat bottom, and $a>0$ is the dimensionless nonlinearity parameter; see \cite[Section~5.1.1.1]{Lannes}, for instance, for details. Throughout, $\partial$ means partial differentiation. Note that the phase speed associated with the linear part of \eqref{E:shallow} is $1$. In other words, the effects of dispersion do not live in \eqref{E:shallow}. On the other hand, the phase speed for water waves is $c_{\rm ww}(\sqrt{b}\xi)$, after normalization of parameters, where
\begin{equation}\label{def:cWW}
c_{\rm ww}^2(\xi)=\frac{\tanh\xi}{\xi}.
\end{equation}
For relatively shallow water or, equivalently, relatively long waves satisfying $b\ll 1$, one may expand the right side of \eqref{def:cWW} and find that 
\begin{equation}\label{appr:cW}
c_{\rm ww}(\sqrt{b}\xi)=1-\frac16b\xi^2+O(b^2).
\end{equation}

But the shallow water theory goes too far. It predicts that {\em all} solutions carrying an increase of elevation break.
Yet observations have long since established that some waves do not break. Perhaps, the neglected dispersion effects inhabit breaking. 

But including some dispersion effects  (see \eqref{appr:cW}), the {\em Korteweg-de Vries (KdV) equation}:
\begin{equation}\label{E:KdV}
\partial_tu+\Big(1+\frac16b\partial_x^2\Big)\partial_xu+\frac32au\partial_xu=0
\end{equation}
goes too far, and it predicts that {\emph no} solutions break. As a matter of fact, the global-in-time well-posedness for \eqref{E:KdV} was established in \cite{CCKTT2003}, for instance, in $H^s(\mathbb{R})$ for $s\geq -3/4$.

To recapitulate, one necessitates some dispersion effects to satisfactorily explain wave breaking, but the dispersion of the KdV equation seems too strong for short waves. It is not surprising because the phase speed $=1-\frac16b\xi^2$ associated with the linear part of \eqref{E:KdV} poorly approximates that for water waves when $b\gg 1$; see \eqref{appr:cW}.


Whitham therefore noted that ``it is intriguing to know what kind of simpler mathematical equation (than the governing equations of the water wave problem) could include" the breaking effects, and he~\cite{Whitham1967} (see also \cite[pp.~477]{Whitham}) put forward
\begin{equation}\label{E:whitham}
\partial_tu+\int^\infty_{-\infty}K(x-y)\partial_yu(y,t)~dy+\frac32au\partial_xu=0,
\end{equation}
where
\begin{equation}\label{def:K}
K(x)=\frac{1}{2\pi}\int^\infty_{-\infty} c(\xi)e^{ix\xi}~d\xi
\end{equation}
and $c(\xi)=c_{\rm ww}(\sqrt{b}\xi)$. That is, $K$ is the Fourier transform of the phase speed for water waves. 
It combines the dispersion relation of water waves and a nonlinearity of the shallow water theory. 
For small amplitude and long waves satisfying $a=O(b^2)$ and $b\ll1$, the Whitham equation agrees with the KdV equation up to the order of $b^2$ during a relevant time scale; see \cite[Section~7.4.5]{Lannes}, for instance, for details. But, including all the dispersion of water waves, it possibly offers improvements over the KdV equation for short waves. Whitham conjectured wave breaking in \eqref{E:whitham}-\eqref{def:K}, where $c=c_{\rm ww}$. 

Seliger~\cite{Seliger} made a rather ingenious argument, albeit formal, and claimed that a sufficiently asymmetric solution of \eqref{E:whitham} would break, provided that $K$ be bounded and integrable, among other hypotheses. Later, Constantin and Escher~\cite{CEbreaking} turned it into a rigorous analytical proof. Unfortunately, the argument does not apply to the Whitham equation, because $c_{\rm ww}$ is not integrable. Naumkin and Shishmar\"ev~\cite{NS} made an alternative argument of wave breaking, provided that $K$ be integrable and $K(x)\leq K_0|x|^{-2/3}$ for $|x|\ll 1$ for some $K_0>0$, among other hypotheses. But $K$ associated with $c_{\rm ww}$ may not be written explicitly. (Moreover, there seem some errors in their arguments.) Here we make a rigorous analytical proof of wave breaking in \eqref{E:whitham}-\eqref{def:K}, where $c=c_{\rm ww}$, whereby we solve Whitham's conjecture. Moreover, we clarify several assertions in \cite{NS} during the course of the proof.
By the way, Whitham~\cite{Whitham} formally argued that $K$ associated with $c_{\rm ww}$ would behave like $|\cdot|^{-1/2}$ near zero and exponentially vanishes at infinity. Recently, Ehrnstr\"om and Wahl\'en~\cite{EW16} analytically confirmed them; see Section~\ref{sec:K} for details. 

\begin{theorem}[Wave breaking in the Whitham equation]\label{thm:whitham}
If $u_0\in H^\infty(\mathbb{R})$ satisfies that 
\begin{equation}\label{AW:m1} 
\epsilon^2(\inf_{x\in\mathbb{R}}u_0'(x))^2>1+\|u_0\|_{H^3(\mathbb{R})}
\end{equation} 
for $\epsilon>0$ sufficiently small, and that
\begin{equation}\label{IW:Gevrey}
\|u_0^{(n)}\|_{L^\infty(\mathbb{R})}\leq ((n-1)g)^{2(n-1)}\qquad\text{for $n=2,3,\dots$}
\end{equation}
for some $g\geq1$, then the solution of the initial value problem associated with \eqref{E:whitham}-\eqref{def:K},
where $c=c_{\rm ww}$, and $u(\cdot,0)=u_0$ exhibits wave breaking for some $T>0$. That is,
\[ 
|u(x,t)|<\infty \qquad \text{for any $x\in\mathbb{R}$}\quad\text{for any $t\in [0,T)$}
\]
but
\[
\inf_{x\in\mathbb{R}}(\partial_xu)(x,t) \to -\infty \qquad\text{as $t\to T-$}.
\] 
Moreover,
\begin{equation}\label{E:T}
-\frac{1}{1+\epsilon}\frac{1}{\inf_{x\in\mathbb{R}}u_0'(x)}<T<
-\frac{1}{(1-\epsilon)^2}\frac{1}{\inf_{x\in\mathbb{R}}u_0'(x)}.
\end{equation} 
\end{theorem}

The Whitham equation is difficult to rigorously justify unless $a=O(b^2)$ and $b\ll1$, whence it is not clear whether the result of Theorem~\ref{thm:whitham} has relevance to the breaking of water waves. Nevertheless, the proof is useful for other related models, which is of independent interest. We illustrate this by discussing KdV equations with fractional dispersion:
\begin{equation}\label{E:fKdV}
\partial_tu+\Lambda^{\alpha-1} \partial_x u+u\partial_xu=0
\end{equation}
in the range $\alpha\geq0$, where $\Lambda=\sqrt{-\partial_x^2}$ is a Fourier multiplier operator, defined as 
\[
\widehat{\Lambda f}(\xi)=|\xi|\widehat{f}(\xi).
\]
In the case of $\alpha=3$, \eqref{E:fKdV} is the KdV equation, after normalization of parameters. In the case of $\alpha=2$, \eqref{E:fKdV} is the Benjamin-Ono equation, and in the case of $\alpha=1$, it is the inviscid Burgers equation. Moreover, in the case of $\alpha=1/2$, the author~\cite{Hur-blowup} observed that \eqref{E:fKdV} shares the dispersion relation and scaling symmetry in common with water waves in the infinite depth. Last but not least, in the case of $\alpha=0$, \eqref{E:fKdV} was proposed in \cite{BH} to model nonlinear waves whose linearized frequency is nonzero but constant. 

Note that $K$ associated with $\Lambda^{\alpha-1}$ in the range $0<\alpha<1$ is $|\cdot|^{-\alpha}$ up to multiplication by a constant (see \eqref{def:Kn} below), not integrable. Hence, the arguments in \cite{CEbreaking} and \cite{NS} do not apply to \eqref{E:fKdV}. Nevertheless, the author~\cite{Hur-blowup} (see also \cite{CCG2010}) proved derivative blowup in \eqref{E:fKdV} for $0<\alpha<1$. Recently, Tao and the author~\cite{HT1} promoted the result to wave breaking for $0<\alpha<1/2$. Here we extend it to $0<\alpha<2/3$.

\begin{theorem}[Wave breaking in \eqref{E:fKdV} for $0<\alpha<2/3$] \label{thm:2/3}
Let $0<\alpha <{\displaystyle \frac23\frac{1-10\epsilon}{1+4\epsilon}}$ for $\epsilon>0$ sufficiently small. 
If $u_0\in H^\infty(\mathbb{R})$ satisfies that 
\begin{align}
\epsilon^2(-\inf_{x\in\mathbb{R}}u_0'(x))^2>&1+\|u_0\|_{H^3(\mathbb{R})}, \label{A3:m1} \\
\epsilon^2(1-\epsilon)^4(-\inf_{x\in\mathbb{R}}u_0'(x))^{3/4}>&\frac{6}{\alpha}
(1+(1+e^{1/\alpha}+e^{1/g})g+g^{1/\alpha}),\label{A3:m2} \\
\epsilon^2(-\inf_{x\in\mathbb{R}}u_0'(x))^{1/4}>&
\frac{e}{1/\alpha-1}\Big(\frac32\Big)^{1/\alpha}\label{A3:m3}
\end{align}
and that
\begin{equation}\label{I:Gevrey}
\|u_0^{(n)}\|_{L^\infty(\mathbb{R})}\leq((n-1)g)^{(n-1)/\alpha}\qquad \text{for $n=2, 3, \dots$}
\end{equation}
for some $g\geq 1$, then the solution of the initial value problem associated with \eqref{E:fKdV}
and $u(\cdot,0)=u_0$ exhibits wave breaking for some $T>0$. Moreover, $T$ satisfies \eqref{E:T}.
\end{theorem}

The hypotheses \eqref{AW:m1} and \eqref{A3:m1}-\eqref{A3:m3} require that $u_0'$ be sufficiently negative somewhere in $\mathbb{R}$. The idea of the proofs lies in that the profile of $u$ steepens until it becomes vertical in finite time. 
The hypotheses \eqref{IW:Gevrey} and \eqref{I:Gevrey} require that $u_0$ belong to the Gevrey class of index $1/\alpha$ (see \cite[pp.~335]{Hormander}, for instance). Since $1/\alpha>1$, nontrivial $u_0$ with compact support exist. They may be removed, provided that $K$ be bounded and integrable; see \cite{CEbreaking}, for instance.

The proofs of Theorem~\ref{thm:whitham} and Theorem~\ref{thm:2/3} follow along the the same line as the argument in \cite{HT1}, examining the ordinary differential equations for the solution and its derivatives of all orders along the characteristics, which by the way involve nonlocal forcing terms. Lemma~\ref{lem:K} indicates that $K$ associated with $c=c_{\rm ww}$ behaves like $|\cdot|^{-1/2}$ near zero. Loosely speaking, therefore, the Whitham equation compares with \eqref{E:fKdV} in the case of $\alpha=1/2$, for which, unfortunately, the argument in \cite{HT1} ceases to apply. Specifically, one loses controls of the second derivative of the solution along the characteristics. We overcome the difficulty by exploiting the ``smoothing effects" of the characteristics, when the derivative of the solution is sufficiently negative; see \eqref{I:X1}, \eqref{claim:X2}, and \eqref{claim:X3} below. 
It enables us to improve the result in \cite{HT1} to $0<\alpha<2/3$. Moreover, we are able to relax the hypotheses in \cite{NS} and \cite{HT1}. 

In recent years, the Whitham equation gathered renewed attention because of its ability to explain short wave phenomena in water. For instance, Ehrnstr\"om and Wahl\'en~\cite{EW16} constructed a highest and cusped, periodic traveling wave of \eqref{E:whitham}-\eqref{def:K}, where $c=c_{\rm ww}$, similarly to the limiting\footnote{Stokes conjectured that the crest of a wave of greatest possible height would exhibit a $120^\circ$ corner. Toland~\cite{Toland1978} made a rigorous existence proof.} Stokes wave. Moreover, Johnson and the author~\cite{HJ2} proved that a small amplitude and periodic traveling wave be spectrally unstable to long wavelength perturbations, provided that the wave number is greater than a critical value, similarly to the Benjamin-Feir instability\footnote{Benjamin and Feir~\cite{BF} and Whitham~\cite{Whitham-BF} discovered that a periodic traveling wave in water would be unstable to slow modulations, provided that the carrier wave number times the undisturbed fluid depth is greater than $1.363\dots$. Bridges and Mielke~\cite{BM1995} made a rigorous proof of spectral instability.} of a Stokes wave. 

One may expect wave breaking in \eqref{E:fKdV} for $0\leq \alpha\leq 1$. As a matter of fact, in the case of $\alpha=1$, \eqref{E:fKdV} is the inviscid Burgers equation, which allows wave breaking. Moreover, in the case of $\alpha=0$, global-in-time weak solutions exist in $L^2(\mathbb{R})\bigcap L^\infty(\mathbb{R})$ (see \cite{BN2014}, for instance), whereas some solutions exhibit derivative blowup (see \cite{CCG2010}, for instance). It is interesting to extend the result of Theorem~\ref{thm:2/3} to $0<\alpha<1$. 
Furthermore, one may expect global regularity (in $H^{\alpha/2}(\mathbb{R})$) for $\alpha>3/2$ but finite time blowup for $0<\alpha<3/2$, because \eqref{E:fKdV} is $L^2$ critical in the case of $\alpha=3/2$. But recent numerical experiments~\cite{KS} suggest that 
the blowup scenario for $1<\alpha< 3/2$  be different from wave breaking. 
It is interesting to analytically confirm finite time blowup for $0\leq \alpha<3/2$, and to understand the blowup scenarios.

\section{Proof of Theorem \ref{thm:2/3}}\label{sec:2/3}

We assume that the initial value problem associated with \eqref{E:fKdV} and $u(\cdot,0)=u_0$ possesses a unique solution in $C^\infty([0,T); H^\infty(\mathbb{R}))$ for some $T>0$. As a matter of fact, one may combine an a priori bound and a compactness argument to work out the local-in-time well-posedness in $H^s(\mathbb{R})$ for $s>3/2$; see \cite{Kato}, for instance, for details. We assume that $T$ is the maximal time of existence.

\

For $x\in\mathbb{R}$, let 
\begin{equation}\label{def:X}
\frac{dX}{dt}(t;x)=u(X(t;x),t) \quad\text{and}\quad X(0;x)=x.
\end{equation}
Since $u(x,t)$ is bounded and satisfies a Lipschitz condition in $x$ for any $x\in\mathbb{R}$ for any $t\in[0,T)$, 
it follows from ODE theory that \eqref{def:X} possesses a unique solution in $C^1([0,T))$ for any $x\in\mathbb{R}$. 
Furthermore, since $u(x,t)$ is smooth in $x$ for any $x\in\mathbb{R}$ for any $t\in[0,T)$, $x\mapsto X(\cdot\,;x)$ is infinitely continuously differentiable throughout the interval $(0,T)$ for any $x\in\mathbb{R}$.

Let 
\begin{equation}\label{def:vn}
v_n(t;x)=(\partial_x^nu)(X(t;x),t) \qquad \text{for $n=0,1,2,\dots$}.
\end{equation}
Differentiating \eqref{E:fKdV} with respect to $x$ and evaluating the result at $X(t;x)$, we arrive at that
\begin{align}
\frac{dv_n}{dt}+\sum_{j=1}^n\Big(\begin{matrix}n\\j\end{matrix}\Big)v_jv_{n+1-j}+&\phi_n(t;x)=0
\qquad \text{for $n=2,3,\dots$},\label{e:vn} \\
\frac{dv_1}{dt}+v_1^2+\phi_1(t;x)&=0, \label{e:v1} 
\intertext{and}
\frac{dv_0}{dt}+\phi_0(t;x)=&0 \label{e:v0}
\end{align}
throughout the interval $(0,T)$ for any $x\in\mathbb{R}$. Here $\Big(\begin{matrix}n\\j\end{matrix}\Big)$ means a binomial coefficient, and 
\begin{align}\label{def:Kn}
\phi_n(t;x)=&(\Lambda^{\alpha-1}\partial_x^{n+1}u)(X(t;x),t) \notag\\
=&\int^\infty_{-\infty}\frac{\text{sgn}(X(t;x)-y)}{|X(t;x)-y|^{1+\alpha}}
((\partial_x^nu)(X(t;x),t)-(\partial_x^nu)(y,t))~dy \notag\\
=&\int^\infty_{-\infty}\frac{\text{sgn}(y)}{|y|^{1+\alpha}}((\partial_x^nu)(X(t;x),t)-(\partial_x^nu)(X(t;x)-y,t))~dy
\end{align}
up to multiplication by a constant for $n=0,1,2,\dots$; see \cite{Hur-blowup}, for instance, for details.
Since $u(x,t)$ is smooth and square integrable in $x$ and smooth in $t$ for any $x\in\mathbb{R}$ for any $t\in[0,T)$, 
and since $X(t;x)$ is continuously differentiable in $t$ and smooth in $x$ for any $t\in[0,T)$ for any $x\in\mathbb{R}$, 
it follows that $\phi_n(t;x)$ is continuously differentiable in $t$ and smooth in $x$ for any $t\in[0,T)$ for any $x\in\mathbb{R}$.

\

For $\delta>0$, we split the integral on the right side of \eqref{def:Kn} and perform an integration by parts to show that
\begin{align}\label{e:Kn}
|\phi_n(t;x)|=&\Big|\Big(\int_{|y|<\delta}+\int_{|y|>\delta}\Big)
\frac{\text{sgn}(y)}{|y|^{1+\alpha}}((\partial_x^nu)(X(t;x),t)-(\partial_x^nu)(X(t;x)-y,t))~dy\Big|\notag \\
\leq&\Big|\frac{1}{\alpha}\delta^{-\alpha}
((\partial_x^nu)(X(t;x)-\delta,t)-(\partial_x^nu)(X(t;x)+\delta,t))\Big| \notag \\
\qquad&+\frac{1}{\alpha}\Big|\int_{|y|<\delta}\frac{1}{|y|^\alpha} (\partial_x^{n+1}u)(X(t;x)-y,t)~dy\Big| \notag \\
\qquad&+\Big|\int_{|y|>\delta}\frac{\text{sgn}(y)}{|y|^{1+\alpha}}
((\partial_x^nu)(X(t;x),t)-(\partial_x^nu)(X(t;x)-y,t))~dy\Big|\notag \\
\leq &\frac{6}{\alpha}(\delta^{-\alpha}\|v_n(t;\cdot)\|_{L^\infty(\mathbb{R})}
+\delta^{1-\alpha}\|v_{n+1}(t;\cdot)\|_{L^\infty(\mathbb{R})})
\end{align}
for $n=0,1,2,\dots$ and for any $t\in [0,T)$ for any $x\in\mathbb{R}$.
Here the second inequality uses that
${\displaystyle \frac{\text{sgn}(y)}{|y|^{1+\alpha}}=-\frac{1}{\alpha}\Big(\frac{1}{|y|^{\alpha}}\Big)'}$
and the last inequality uses that $0<\alpha<2/3$. 

Let
\begin{equation}\label{def:q}
m(t)=\inf_{x\in\mathbb{R}}v_1(t;x)=\inf_{x\in\mathbb{R}}(\partial_xu)(x,t)=:m(0)q^{-1}(t).
\end{equation}
Note that $v_1(t;\cdot)$ and, hence, $m(t)$ are continuous for any $t\in[0,T)$. 
Note moreover that $m(t)<0$, $q(0)=1$ and $q(t)>0$ for any $t\in[0,T)$. 
Indeed, if $m(t)\geq 0$ for some $t\in [0,T)$ then $u(\cdot,t)$ must be nondecreasing in $\mathbb{R}$, whence $u(\cdot,t)\equiv0$. 

\

We shall show that 
\begin{equation}\label{claim:A}
|\phi_1(t;x)|<\epsilon^2m^2(t)\qquad \text{for any $t\in [0,T)$}\quad\text{for any $x\in \mathbb{R}$}.
\end{equation}
It follows from \eqref{A3:m1} and the Sobolev inequality that 
\[
|\phi_1(0;x)|=|\Lambda^{\alpha-1}u_0''(x)|\leq \|u_0\|_{H^{\alpha+3/2+}(\mathbb{R})}<\epsilon^2m^2(0)
\qquad\text{for any $x\in\mathbb{R}$}.
\]
In other words, \eqref{claim:A} holds at $t=0$.
Suppose on the contrary that $|\phi_1(T_1;x)|=\epsilon^2m^2(T_1)$ 
for some $T_1\in(0,T)$ for some $x\in\mathbb{R}$. By continuity, we may assume that
\begin{equation}\label{I:A}
|\phi_1(t;x)|\leq \epsilon^2m^2(t) \qquad \text{for any $t\in [0,T_1]$}\quad\text{for any $x\in \mathbb{R}$}.
\end{equation}
We seek a contradiction.

\begin{lemma}\label{lem:S}
For $0<\gamma<1$ and for $t\in[0,T_1]$, let 
\begin{equation}\label{def:Sigma}
\Sigma_{\gamma}(t)=\{x\in\mathbb{R}:v_1(t;x)\leq (1-\gamma)m(t)\}.
\end{equation}
If $0<\epsilon\leq\gamma<1/2$ for $\epsilon>0$ sufficiently small
then $\Sigma_\gamma(t_2)\subset\Sigma_\gamma(t_1)$ whenever $0\leq t_1\leq t_2\leq T_1$.
\end{lemma}

The proof extends that of \cite[Lemma~2.1]{HT1}. We present the details in Appendix~\ref{app:proofs} for completeness. 

\begin{lemma}\label{lem:q}
$0<q(t)\leq 1$ and it is decreasing for any $t\in[0,T_1]$.
\end{lemma}

\begin{proof}
The proof is very similar to that of \cite[Lemma~2.2]{HT1}. Here we include the details for future usefulness.

Let $x\in\Sigma_\gamma(T_1)$, where $0<\epsilon\leq\gamma<1/2$ for $\epsilon>0$ sufficiently small. 
We suppress it for simplicity of notation.
Note from \eqref{def:q} and Lemma~\ref{lem:S} that 
\begin{equation}\label{I:mv}
m(t)\leq v_1(t)\leq (1-\gamma)m(t)(<0)\qquad \text{for any $t\in[0,T_1]$}.
\end{equation}
Let's write \eqref{e:v1} as
\begin{equation}\label{def:r}
v_1(t)=\frac{v_1(0)}{1+v_1(0)\int^t_0(1+(v_1^{-2}\phi_1)(\tau))~d\tau}=:m(0)r^{-1}(t).
\end{equation}
Clearly, $r(t)>0$ for any $t\in[0,T_1]$. Note from \eqref{I:mv} and \eqref{I:A} that
\[
|(v_1^{-2}\phi_1)(t)|<(1-\gamma)^{-2}\epsilon^2<\epsilon\qquad\text{for any $t\in[0,T_1]$}
\]
for $\epsilon>0$ sufficiently small. Therefore, it follows from \eqref{def:r} that 
\begin{equation}\label{I:dr/dt}
(1+\epsilon)m(0)\leq \frac{dr}{dt}\leq (1-\epsilon)m(0)\qquad\text{throughout the interval $(0,T_1)$}.
\end{equation}
Consequently, $r(t)$ and, hence, $v_1(t)$ (see \eqref{def:r}) are decreasing for any $t\in[0,T_1]$. 
Furthermore, $m(t)$ and, hence, $q(t)$ (see \eqref{def:q}) are decreasing for any $t\in[0,T_1]$.
This completes the proof. It follows from \eqref{def:q}, \eqref{def:r} and \eqref{I:mv} that
\begin{equation}\label{I:qr}
q(t)\leq r(t)\leq \frac{1}{1-\gamma}q(t)\qquad\text{for any $t\in[0,T_1]$}.
\end{equation} 
\end{proof}

\begin{lemma}\label{lem:qs}
For $s>0$, $s\neq1$, and for $t\in[0,T_1]$,
\begin{equation}\label{I:s>1}
\int^t_0q^{-s}(\tau)~d\tau\leq
-\frac{1}{s-1}\frac{1}{(1-\epsilon)^{1+s}}\frac{1}{m(0)}\Big(q^{1-s}(t)-\frac{1}{(1-\epsilon)^{1-s}}\Big).
\end{equation}
\end{lemma}

The proof is found in \cite[Lemma~2.3]{HT1}, for instance. Hence, we omit the details. 
See instead the proof of \eqref{I:s>1'} below. 


\

We shall show that
\begin{align}
\|v_0(t;\cdot)\|_{L^\infty(\mathbb{R})}&=\|u(\cdot,t)\|_{L^\infty(\mathbb{R})}<C_0, \label{claim:v0} \\
\|v_1(t;\cdot)\|_{L^\infty(\mathbb{R})}&=\|(\partial_xu)(\cdot,t)\|_{L^\infty(\mathbb{R})}<C_1q^{-1}(t), \label{claim:v1}  \\
\|v_n(t;\cdot)\|_{L^\infty(\mathbb{R})}&=\|(\partial_x^nu)(\cdot,t)\|_{L^\infty(\mathbb{R})}<
C_2((n-1)g)^{(n-1)/\alpha}q^{-1-(n-1)\sigma}(t)\label{claim:vn}
\end{align}
for $n=2,3,\dots$ for any $t\in[0,T_1]$, where 
\begin{equation}\label{def:C}
C_0=2(\|u_0\|_{L^\infty(\mathbb{R})}+\|u_0'\|_{L^\infty(\mathbb{R})}),\quad C_1=2\|u_0'\|_{L^\infty(\mathbb{R})}, 
\quad C_2=(-m(0))^{3/4}
\end{equation}
and
\begin{equation}\label{I:sigma}
\sigma=\frac32+6\epsilon\quad\text{so that}\quad \sigma\alpha<1-10\epsilon
\end{equation}
for $\alpha$ and $\epsilon$ in Theorem~\ref{thm:2/3}. Note from \eqref{A3:m1} that 
\[
\frac12C_1=\|u_0'\|_{L^\infty(\mathbb{R})}>C_2>1,
\]
and we tacitly exercise it throughout the proof. It follows from \eqref{def:C}, \eqref{def:q} and \eqref{I:Gevrey}, \eqref{A3:m1} that
\begin{align*}
\|v_0(0;\cdot)\|_{L^\infty(\mathbb{R})}=&\|u_0\|_{L^\infty(\mathbb{R})}<C_0,\\
\|v_1(0;\cdot)\|_{L^\infty(\mathbb{R})}=&\|u_0'\|_{L^\infty(\mathbb{R})}<C_1=C_1q^{-1}(0), \\
\|v_n(0;\cdot)\|_{L^\infty(\mathbb{R})}=&\|u_0^{(n)}\|_{L^\infty(\mathbb{R})}
<C_2((n-1)g)^{(n-1)/\alpha}q^{-1-(n-1)\sigma}(0)
\end{align*}
for $n=2,3,\dots$. 
In other words, \eqref{claim:v0}, \eqref{claim:v1} and \eqref{claim:vn} hold for any $n=0, 1, 2, \dots$ at $t=0$.
Suppose on the contrary that \eqref{claim:v0}, \eqref{claim:v1} and \eqref{claim:vn} hold
for any $n=0, 1, 2, \dots$ throughout the interval $[0,T_2)$ 
but do not for some $n\geq0$ at $t=T_2$ for some $T_2 \in (0,T_1]$. By continuity, we find that
\begin{align}
\|v_0(t;\cdot)\|_{L^\infty(\mathbb{R})}\leq &C_0, \label{I:v0}\\ 
\|v_1(t;\cdot)\|_{L^\infty(\mathbb{R})}\leq &C_1q^{-1}(t), \label{I:v1}\\
\|v_n(t;\cdot)\|_{L^\infty(\mathbb{R})}\leq &C_2((n-1)g)^{(n-1)/\alpha}q^{-1-(n-1)\sigma}(t) \label{I:vn}
\end{align}
for $n=2,3,\dots$ for any $t\in[0,T_2]$. We seek a contradiction.

\subsection*{Proof of \eqref{claim:v0}}
The proof is similar to that in \cite{HT1}. Here we include the details for future usefulness.

It follows from \eqref{e:Kn}, where $\delta(t)=q(t)$, and \eqref{I:v0}, \eqref{I:v1} that 
\begin{equation}\label{I:K0}
|\phi_0(t;x)|\leq \frac{6}{\alpha}(C_0q^{-\alpha}(t)+C_1q^{1-\alpha}(t)q^{-1}(t))
=\frac{6}{\alpha}(C_0+C_1)q^{-\alpha}(t)
\end{equation}
for any $t\in[0,T_2]$ for any $x\in\mathbb{R}$. 
Integrating \eqref{e:v0} over the interval $[0,T_2]$, we then show that 
\begin{align*}
|v_0(T_2;x)|\leq& \|u_0\|_{L^\infty(\mathbb{R})}+\int^{T_2}_0|\phi_0(t;x)|~dt \\
\leq&\frac12C_0+\frac{6}{\alpha}(C_0+C_1)\int^{T_2}_0q^{-\alpha}(t)~dt \\
\leq&\frac12C_0-\frac{6}{\alpha}(C_0+C_1)\frac{1}{1-\alpha}
\frac{1}{(1-\epsilon)^{1+\alpha}}\frac{1}{m(0)}\Big(\frac{1}{(1-\epsilon)^{1-\alpha}}-q^{1-\alpha}(T_2)\Big) \\
<&\frac12C_0-\frac{6}{\alpha(1-\alpha)}(C_0+C_1)\frac{1}{(1-\epsilon)^2}\frac{1}{m(0)}\\
<&C_0
\end{align*}
for any $x\in\mathbb{R}$.
Therefore, \eqref{claim:v0} holds throughout the interval $[0,T_2]$.
Here the second inequality uses \eqref{def:C} and \eqref{I:K0}, the third inequality uses \eqref{I:s>1}, 
the fourth inequality uses Lemma~\ref{lem:q}, and the last inequality uses that \eqref{A3:m2} implies that
\[
-(1-\epsilon)^2m(0)>\frac{12}{\alpha(1-\alpha)}\Big(1+\frac{C_1}{C_0}\Big)
\]
for $\epsilon>0$ sufficiently small. Indeed, $1/\alpha>3/2$, $1<1/(1-\alpha)<3$, $m(0)<-1$, and $g\geq1$ by hypotheses, and  $C_1/C_0<1$ by \eqref{def:C}.

\subsection*{Proof of \eqref{claim:v1}}
The proof is similar to that in \cite{HT1}. Here we include the details for future usefulness.

It follows from \eqref{e:Kn}, where $\delta(t)=q^\sigma(t)$, and \eqref{I:v1}, \eqref{I:vn} that
\begin{align}
|\phi_1(t;x)|\leq&\frac{6}{\alpha}(C_1q^{-1}q^{-\sigma\alpha}(t)
+C_2g^{1/\alpha}q^{\sigma-\sigma\alpha}(t)q^{-1-\sigma}(t)) \notag \\
=&\frac{6}{\alpha}(C_1+C_2g^{1/\alpha})q^{-1-\sigma\alpha}(t) \label{I:K1}
\end{align}
for any $t\in[0,T_2]$ for any $x\in\mathbb{R}$. Suppose for now that $v_1(T_2;x)\geq 0$.
Note from \eqref{e:v1} that 
\[
\frac{dv_1}{dt}(t;x)=-v_1^2(t;x)-\phi_1(t;x)\leq |\phi_1(t;x)|
\]
for any $t\in (0,T_2)$ for any $x\in\mathbb{R}$. Integrating this over the interval $[0,T_2]$, we then show that
\begin{align*}
v_1(T_2;x)\leq& \|u_0'\|_{L^\infty(\mathbb{R})}+\int^{T_2}_0|\phi_1(t;x)|~dt \\
\leq&\frac12C_1+\frac{6}{\alpha}(C_1+C_2g^{1/\alpha})\int^{T_2}_0q^{-2}(t)~dt \\
\leq&\frac12C_1-\frac{6}{\alpha}(C_1+C_2g^{1/\alpha})
\frac{1}{(1-\epsilon)^3}\frac{1}{m(0)}(q^{-1}(T_2)-(1-\epsilon)) \\
<&\frac12C_1q^{-1}(T_2)-\frac{6}{\alpha}(C_1+C_2g^{1/\alpha})
\frac{1}{(1-\epsilon)^3}\frac{1}{m(0)}q^{-1}(T_2)\\ 
<&C_1q^{-1}(T_2).
\end{align*}
The second inequality uses \eqref{def:C} and \eqref{I:K1}, Lemma~\ref{lem:q}, \eqref{I:sigma}, 
the third inequality uses \eqref{I:s>1}, the fourth inequality uses Lemma~\ref{lem:q}, 
and the last inequality uses that \eqref{A3:m2} implies that
\[
-(1-\epsilon)^3m(0)>\frac{12}{\alpha}\Big(1+\frac{C_2}{C_1}g^{1/\alpha}\Big)
\]
for $\epsilon>0$ sufficiently small. Indeed, $m(0)<-1$ by hypotheses and $C_2/C_1<1/2$ by \eqref{def:C}. 

\

Suppose on the other hand that $v_1(T_2;x)<0$. We may assume without loss of generality that 
$\|u_0'\|_{L^\infty(\mathbb{R})}=-m(0)$; we take $-u$ otherwise. It then follows from \eqref{def:q} and \eqref{def:C} that 
\[
v_1 (T_2;x) \geq m(T_2)= m(0) q^{-1}(T_2)> -C_1 q^{-1}(T_2).
\]
Therefore, \eqref{claim:v1} holds throughout the interval $[0,T_2]$.

\subsection*{Proof of \eqref{claim:vn} for $n\geq 3$}
The proof is similar to that in \cite{HT1}. Here we include the details for future usefulness. 

For $n\geq 2$, it follows from \eqref{e:Kn}, where $\delta(t)=(ng)^{-1/\alpha}q^\sigma(t)$, and \eqref{I:vn} that
\begin{align}\label{I:K2}
|\phi_n(t;x)|\leq &\frac{6}{\alpha}(ngC_2((n-1)g)^{(n-1)/\alpha}q^{-\sigma\alpha}(t)q^{-1-(n-1)\sigma}(t)\notag\\
&\quad+(ng)^{1-1/\alpha}C_2(ng)^{n/\alpha}q^{\sigma-\sigma\alpha}q^{-1-n\sigma}(t)) \notag \\
=&\frac6\alpha ngC_2((n-1)g)^{(n-1)/\alpha}\Big(1+\Big(\frac{n}{n-1}\Big)^{(n-1)/\alpha}\Big)
q^{-1-\sigma\alpha-(n-1)\sigma}(t)\notag  \\
<&\frac6\alpha(1+e^{1/\alpha})ngC_2((n-1)g)^{(n-1)/\alpha}q^{-1-\sigma\alpha-(n-1)\sigma}(t)
\end{align}
for any $t\in[0,T_2]$ for any $x\in\mathbb{R}$. 

For $n\geq 2$, furthermore, let 
\begin{equation}\label{I:mv'}
v_1(T_{3,n};x)=m(T_{3,n})\quad\text{and}\quad
m(t)\leq v_1(t;x)\leq \frac{1}{(1+\epsilon)^{1/(2+(n-1)\sigma)}}m(t)
\end{equation}
for any $t\in[T_{3,n},T_2]$, for some $T_{3,n}\in(0,T_2)$ and for some $x\in\mathbb{R}$. Indeed, since $v_1$ and $m$ are uniformly continuous throughout the interval $[0, T_2]$, we may find $T_{3,n}$ close to $T_2$ so that \eqref{I:mv'} holds. Of course, $x$ depends on $n$, but we suppress it for simplicity of notation. We rerun the argument in the proof of Lemma~\ref{lem:q} to arrive at that
\begin{equation}\label{I:dr/dt'}
(1+\epsilon)m(0)\leq\frac{dr}{dt}\leq(1-\epsilon)m(0)\qquad \text{throughout the interval $(T_{3,n}, T_2)$}
\end{equation}
for $\epsilon>0$ sufficiently small, and that
\begin{equation}\label{I:qr'}
q(t)\leq r(t)\leq (1+\epsilon)^{1/(2+(n-1)\sigma)}q(t)\qquad\text{for any $t\in[T_{3,n},T_2]$.}
\end{equation}
It then follows that
\begin{align}
\int^{T_2}_{T_{3,n}}&q^{-2-(n-1)\sigma}(t)~dt \notag\\ 
&\leq(1+\epsilon)\int^{T_2}_{T_{3,n}}r^{-2-(n-1)\sigma}(t)~dt \notag\\
&\leq\frac{1+\epsilon}{1-\epsilon}\frac{1}{m(0)}
\int^{T_2}_{T_{3,n}}r^{-2-(n-1)\sigma}(t)\frac{dr}{dt}(t)~dt \notag \\
&=-\frac{1}{1+(n-1)\sigma}\frac{1+\epsilon}{1-\epsilon}\frac{1}{m(0)}
(r^{-1-(n-1)\sigma}(T_2)-r^{-1-(n-1)\sigma}(T_{3,n})) \notag \\
&\leq-\frac{1}{1+(n-1)\sigma}\frac{1+\epsilon}{1-\epsilon}\frac{1}{m(0)}
(q^{-1-(n-1)\sigma}(T_2)-q^{-1-(n-1)\sigma}(T_{3,n})).\label{I:s>1'}
\end{align}
This offers a refinement over \eqref{I:s>1} when $T_{3,n}$ and $T_2$ are close. Observe the right side of \eqref{I:s>1'} decreases in $n$. Here the first inequality uses \eqref{I:qr'}, the second inequality uses \eqref{I:dr/dt'}, and the last inequality uses \eqref{I:qr'} and \eqref{I:mv'}.

\begin{lemma}\label{lem:n3}
For $n\geq 3$, 
\begin{equation}\label{I:n3}
\sum_{j=2}^{n-1}\Big(\begin{matrix}n\\j\end{matrix}\Big)(j-1)^{(j-1)/\alpha}(n-j)^{(n-j)/\alpha}
\leq \frac{e}{1/\alpha-1}\Big(\frac32\Big)^{1/\alpha-1}n(n-1)^{(n-1)/\alpha}.
\end{equation}
\end{lemma}

The proof is in \cite[Lemma~2.6.1]{NS}, for instance.
We include the details in Appendix~\ref{app:proofs} for completeness.

\

{For $n\geq 3$}, let $|v_n(T_2;x_n)|={\displaystyle \max_{x\in\mathbb{R}}|v_n(T_2;x)|}$. 
We may assume without loss of generality that $v_n(T_2;x_n)>0$; we take $-u$ otherwise. We choose $T_{3,n}$ close to $T_2$ so that 
\begin{equation}\label{I:vn>0}
v_n(t;x_n)\geq 0\qquad \text{for any $t\in[T_{3,n},T_2]$}.
\end{equation}
We necessarily choose $T_{3,n}$ closer to $T_2$ so that \eqref{I:mv'} holds for some $x\in\mathbb{R}$. 
Consequently, \eqref{I:s>1'} holds. It follows from \eqref{e:vn} that 
\begin{align*}
\frac{dv_n}{dt}&(t;x_n)\\=&
-(n+1)v_1(t;x_n)v_n(t;x_n)
-\sum_{j=2}^{n-1}\Big(\begin{matrix}n\\j\end{matrix}\Big)v_j(\cdot\;;x_n)v_{n+1-j}(t;x_n)
-\phi_n(t;x_n) \\
\leq &-(n+1)m(0)C_2((n-1)g)^{(n-1)/\alpha}q^{-1}(t)q^{-1-(n-1)\sigma}(t) \\
&+\sum_{j=2}^{n-1}\Big(\begin{matrix}n\\j\end{matrix}\Big)
C_2^2((j-1)g)^{(j-1)/\alpha}((n-j)g)^{(n-j)/\alpha}q^{-1-(j-1)\sigma}(t)q^{-1-(n-j)\sigma}(t)\\
&+|\phi_n(t;x_n)| \\
<&-m(0)(n+1)C_2((n-1)g)^{(n-1)/\alpha}q^{-2-(n-1)\sigma}(t)\\
&+\frac{e^{1/\alpha}}{1/\alpha-1}\Big(\frac32\Big)^{1/\alpha}n
C_2^2((n-1)g)^{(n-1)/\alpha}q^{-2-(n-1)\sigma}(t)\\
&+\frac{6}{\alpha}(1+e^{1/\alpha})ngC_2((n-1)g)^{(n-1)/\alpha}q^{-1-\sigma\alpha-(n-1)\sigma}(t)\\
\leq &\Big(-m(0)(n+1)+\frac{e}{1/\alpha-1}\Big(\frac32\Big)^{1/\alpha-1}C_2n
+\frac{6}{\alpha}(1+e^{1/\alpha})ng\Big)\\ 
&\hspace*{150pt} \times C_2((n-1)g)^{(n-1)/\alpha}q^{-2-(n-1)\sigma}(t)
\end{align*}
for any $t \in (T_{3,n},T_2)$. 
The first inequality uses \eqref{def:q}, \eqref{I:vn>0} and \eqref{I:vn}, 
the second inequality uses \eqref{I:n3} and \eqref{I:K2},
and the last inequality uses Lemma~\ref{lem:q} and \eqref{I:sigma}. 
Integrating this over the interval $[T_{3,n},T_2]$, we then show that 
\begin{align*}
v_n(T_2;x_n)< &v_n(T_{3,n};x_n)\\ &+\Big(-m(0)(n+1)
+\frac{e}{1/\alpha-1}\Big(\frac32\Big)^{1/\alpha-1}C_2n+\frac{6}{\alpha}(1+e^{1/\alpha})ng\Big) \\
&\hspace{110pt}\times C_2((n-1)g)^{(n-1)/\alpha}\int^{T_2}_{T_{3,n}}q^{-2-(n-1)\sigma}(t)~dt \\
\leq&C_2((n-1)g)^{(n-1)/\alpha}q^{-1-(n-1)\sigma}(T_{3,n}) \\ &-\Big(-m(0)(n+1)
+\frac{e}{1/\alpha-1}\Big(\frac32\Big)^{1/\alpha-1}C_2n+\frac{6}{\alpha}(1+e^{1/\alpha})ng\Big)\\
&\hspace*{80pt}\times\frac{1}{1+(n-1)\sigma}\frac{1+\epsilon}{1-\epsilon}\frac{1}{m(0)}
C_2((n-1)g)^{(n-1)/\alpha}\\ 
&\hspace*{130pt}\times (q^{-1-(n-1)\sigma}(T_2)-q^{-1-(n-1)\sigma}(T_{3,n})) \\
<&C_2((n-1)g)^{(n-1)/\alpha}q^{-1-(n-1)\sigma}(T_{3,n}) \\
&+\frac{n+1+\epsilon n}{1+(n-1)\sigma}\frac{1+\epsilon}{1-\epsilon}C_2((n-1)g)^{(n-1)/\alpha} 
(q^{-1-(n-1)\sigma}(T_2)-q^{-1-(n-1)\sigma}(T_{3,n})) \\
<&\Big(1-\frac{4+3\epsilon}{2\sigma+1}\frac{1+\epsilon}{1-\epsilon}\Big)
C_2((n-1)g)^{(n-1)/\alpha}q^{-1-(n-1)\sigma}(T_{3,n}) \\
&+\frac{4+3\epsilon}{2\sigma+1}\frac{1+\epsilon}{1-\epsilon}
C_2((n-1)g)^{(n-1)/\alpha}q^{-1-(n-1)\sigma}(T_2) \\
<&C_2((n-1)g)^{(n-1)/\alpha}q^{-1-(n-1)\sigma}(T_2).
\end{align*}
Therefore, \eqref{claim:vn} holds for $n=3, 4, \dots$ throughout the interval $[0,T_2]$. 
Here the second inequality uses \eqref{I:vn} and \eqref{I:s>1'}, 
the third inequality uses that \eqref{A3:m2} and \eqref{A3:m3} imply that
\[
-\epsilon m(0)>\frac{e}{1/\alpha-1}\Big(\frac32\Big)^{1/\alpha-1}C_2+\frac6\alpha(1+e^{1/g})g
\]
for $\epsilon>0$ sufficiently small. Indeed, $m(0)<-1$ by hypotheses and recall \eqref{def:C}. 
The fourth inequality uses \eqref{I:sigma} and that 
${\displaystyle \frac{(1+\epsilon)n+1}{n\sigma+1-\sigma}}$ deceases in $n\geq 3$,
and the last inequality uses \eqref{I:sigma} and Lemma~\ref{lem:q}. Indeed,
\[
0<\frac{4+3\epsilon}{2\sigma+1}\frac{1+\epsilon}{1-\epsilon}<1.
\]

\subsection*{Proof of \eqref{claim:vn} for $n=2$ when $x_2\notin \Sigma_{1/3}(T_2)$}
Let $|v_2(T_2;x_2)|={\displaystyle \max_{x\in\mathbb{R}}|v_2(T_2;x)|}$. 
We may assume without loss of generality that $v_2(T_2;x_2)>0$.
We choose $T_{3,n}$ close to $T_2$ so that 
\begin{equation}\label{I:v2>0}
v_2(t;x_2)\geq 0 \qquad \text{for any $t\in[T_{3,n},T_2]$}.
\end{equation}
We necessarily choose $T_{3,n}$ closer to $T_2$ so that \eqref{I:mv'} and, hence, \eqref{I:s>1'} hold.

Suppose for now that $x_2\notin\Sigma_{1/3}(T_2)$, i.e. $v_1(T_2;x_2)>\frac23m(T_2)$
(see \eqref{def:Sigma}). We may necessarily choose $T_{3,n}$ closer to $T_2$ so that 
\begin{equation}\label{I:1/3}
v_1(t;x_2)\geq \frac23m(t)\qquad \text{for any $t\in[T_{3,n},T_2]$.}
\end{equation}
Indeed, $v_1$ and $m$ are uniformly continuous throughout the interval $[0,T_2]$. 
The proof is similar to that for $n\geq 3$. Specifically, it follows from \eqref{e:vn} that 
\begin{align*}
\frac{dv_2}{dt}(t;x_2)=&-3v_1(t;x_2)v_2(t;x_2)-\phi_2(t;x_2)\\
\leq&-2m(0)C_2g^{1/\alpha}q^{-1}(t)q^{-1-\sigma}(t)
+\frac6\alpha(1+e^{1/\alpha})2gC_2g^{1/\alpha}q^{-1-\sigma\alpha-\sigma}(t) \\
\leq&2\Big(-m(0)+\frac{6}{\alpha}(1+e^{1/\alpha})g\Big)C_2g^{1/\alpha}q^{-2-\sigma}(t)
\end{align*}
for any $t\in(T_{3,n},T_2)$.
The first inequality uses \eqref{I:1/3}, \eqref{I:v2>0}, \eqref{I:mv} and \eqref{I:K2},
and the second inequality uses Lemma~\ref{lem:q} and \eqref{I:sigma}.
Integrating this over the interval $[T_{3,n},T_2]$, we then show that
\begin{align*}
v_2(T_2;x_2)< &v_2(T_{3,n};x_2)
+2\Big(-m(0)+\frac{6}{\alpha}(1+e^{1/\alpha})g\Big)C_2g^{1/\alpha}\int^{T_2}_{T_{3,n}}q^{-2-\sigma}(t)~dt\\
\leq & C_2g^{1/\alpha}q^{-1-\sigma}(T_{3,n}) 
-2\Big(-m(0)+\frac{6}{\alpha}(1+e^{1/\alpha})g\Big)\\
&\hspace*{85pt}\times \frac{1}{1+\sigma}\frac{1+\epsilon}{1-\epsilon}\frac{1}{m(0)}
C_2g^{1/\alpha}(q^{-1-\sigma}(T_2)-q^{-1-\sigma}(T_{3,n})) \\
\leq &C_2g^{1/\alpha}q^{-1-\sigma}(T_{3,n})
+\frac{2}{1+\sigma}\frac{(1+\epsilon)^2}{1-\epsilon}
C_2g^{1/\alpha}(q^{-1-\sigma}(T_2)-q^{-1-\sigma}(T_{3,n})) \\
=&\Big(1-\frac{2}{1+\sigma}\frac{(1+\epsilon)^2}{1-\epsilon}\Big)C_2g^{1/\alpha}q^{-1-\sigma}(T_{3,n})
+\frac{2}{1+\sigma}\frac{(1+\epsilon)^2}{1-\epsilon}C_2g^{1/\alpha}q^{-1-\sigma}(T_2)\\
\leq&C_2g^{1/\alpha}q^{-1-\sigma}(T_2).
\end{align*}
The second inequality uses \eqref{I:vn} and \eqref{I:s>1'}, 
and the third inequality uses that \eqref{A3:m2} implies that
\[
-\epsilon m(0)>\frac{6}{\alpha}(1+e^{1/\alpha}g)
\]
for $\epsilon>0$ sufficiently small. Indeed, $m(0)<-1$ by hypotheses. The last inequality uses \eqref{I:sigma} and Lemma~\ref{lem:q}. Indeed,
\[
0<\frac{2}{1+\sigma}\frac{(1+\epsilon)^2}{1-\epsilon}<1.
\]

\subsection*{Proof of \eqref{claim:vn} for $n=2$ $x_2\in \Sigma_{1/3}(T_2)$}
Suppose on the other hand that $x_2\in\Sigma_{1/3}(T_2)$. It follows from Lemma~\ref{lem:S} that 
\begin{equation}\label{I:2/3}
v_1(t;x_2)\leq\frac23m(t)<0\qquad \text{for any $t\in[0,T_2]$.}
\end{equation}
We shall explore the ``smoothing effects" of the solution of \eqref{def:X}.

Differentiating \eqref{def:X} with respect to $x$ and recalling \eqref{def:vn}, we arrive at that 
\begin{alignat}{2}
\frac{d}{dt}(\partial_xX)=&v_1(\partial_xX), \qquad &&(\partial_xX)(0;x)=1\label{e:X1}
\intertext{and}
\frac{d}{dt}(\partial_x^2X)=&v_2(\partial_xX)^2+v_1(\partial_x^2X), 
\qquad &&(\partial_x^2X)(0;x)=0, \label{e:X2} \\
\frac{d}{dt}(\partial_x^3X)=&v_3(\partial_xX)^3
+3v_2(\partial_xX)(\partial_x^2X)+v_1(\partial_x^3X), \qquad &&(\partial_x^3X)(0;x)=0\label{e:X3}
\end{alignat}
throughout the interval $(0,T_2)$. Integrating \eqref{e:v0}, moreover, we show that 
\[
v_0(t;x)=u_0(x)-\int^t_0\phi_0(\tau;x)~d\tau
\]
for any $t\in[0,T_2]$ for any $x\in\mathbb{R}$. Differentiating it with respect to $x$ and recalling \eqref{def:vn}, we then arrive at that
\begin{align}
(v_2(\partial_xX)^2+v_1(\partial_x^2X))(t;x)\qquad \qquad=u_0''(x)-I_2(t;x),\label{e:v2I}\\
(v_3(\partial_xX)^3+3v_2(\partial_xX)(\partial_x^2X)+v_1(\partial_x^3X))(t;x)=u_0'''(x)-I_3(t;x) \label{e:v3I}
\end{align}
for any $t\in[0,T_2]$ for any $x\in\mathbb{R}$, where
\begin{align}
I_2(t;x)=&\int^t_0(\phi_2(\partial_xX)^2+\phi_1(\partial_x^2X))(\tau;x)~d\tau,\label{def:I2} \\
I_3(t;x)=&\int^t_0(\phi_3(\partial_xX)^3+3\phi_2(\partial_xX)(\partial_x^2X)+\phi_1(\partial_x^3X))(\tau;x)~d\tau.
\label{def:I3}
\end{align}
Note from \eqref{e:X3} and \eqref{e:v3I} that 
\begin{equation}\label{e:X3'}
\frac{d}{dt}(\partial_x^3X)(\cdot\;;x)=u_0'''(x)-I_3(\cdot\;;x)\quad\text{and}\quad(\partial_x^3X)(0;x)=0.
\end{equation}

\

We claim that 
\begin{equation}\label{I:X1}
\frac12q^{1+2\epsilon}(t)\leq (\partial_xX)(t;x_2)\leq 2q^{1-\epsilon}(t)\qquad\text{for any $t\in[0,T_2]$}.
\end{equation}
Indeed, it follows from \eqref{def:X}, \eqref{e:X1} and \eqref{def:r}, \eqref{I:dr/dt} that 
\[
\frac{1}{1-\epsilon}\frac{dr/dt}{r}\leq 
\frac{d(\partial_xX)/dt}{\partial_xX}\leq\frac{1}{1+\epsilon}\frac{dr/dt}{r}
\]
throughout the interval $(0,T_2)$. Integrating this over the interval $[0,t]$ and recalling \eqref{e:X1}, 
we then show that
\[
\Big(\frac{r(t)}{r(0)}\Big)^{1/(1-\epsilon)}\leq (\partial_xX)(t;x_2)\leq \Big(\frac{r(t)}{r(0)}\Big)^{1/(1+\epsilon)}
\]
for any $t\in [0,T_2]$. Therefore \eqref{I:X1} follows from \eqref{I:qr}. 

\

To proceed, we shall show that
\begin{align}
|(\partial_x^2X)(t;x_2)|<&-\frac{8}{m(0)}C_2g^{1/\alpha}q^{2-\sigma-2\epsilon}(t) \label{claim:X2}
\intertext{and}
|(\partial_x^3X)(t;x_2)|<&\frac{\epsilon}{m^2(0)}C_2^2(2g)^{2/\alpha}q^{3-2\sigma+7\epsilon}(t)\label{claim:X3}
\end{align}
for any $t\in[0,T_2]$. 
It follows from \eqref{e:X2} and \eqref{e:X3} that \eqref{claim:X2} and \eqref{claim:X3} hold at $t=0$. 
Suppose on the contrary that \eqref{claim:X2} and \eqref{claim:X3} hold throughout the interval $[0,T_4)$ 
but do not at $t=T_4$ for some $T_4\in(0,T_2]$. By continuity, we find that
\begin{align}
|(\partial_x^2X)(t;x_2)|\leq &-\frac{8}{m(0)}C_2g^{1/\alpha}q^{2-\sigma-2\epsilon}(t) \label{I:X2}
\intertext{and}
|(\partial_x^3X)(t;x_2)|\leq &\frac{\epsilon}{m^2(0)}C_2^2(2g)^{2/\alpha}q^{3-2\sigma+7\epsilon}(t) \label{I:X3}
\end{align}
for any $t\in[0,T_4]$. We seek a contradiction.

\

We use \eqref{def:I2} to compute that 
\begin{align}
|I_2(t;x_2)|\leq& \int^t_0\Big(4\frac{6}{\alpha}(1+e^{1/\alpha})2gC_2g^{1/\alpha}
q^{-1-\sigma\alpha-\sigma}(\tau)q^{2-2\epsilon}(\tau) \notag\\
&\quad\quad-8\frac{6}{\alpha}(C_1+C_2g^{1/\alpha})\frac{1}{m(0)}C_2g^{1/\alpha}
q^{-1-\sigma\alpha}(\tau)q^{2-\sigma-2\epsilon}(\tau)\Big)~d\tau \notag\\
\leq&\frac{48}{\alpha}\Big((1+e^{1/\alpha})g+2\Big(1+\frac{C_2}{C_1}g^{1/\alpha}\Big)\Big)
C_2g^{1/\alpha}\int^t_0 q^{-\sigma+8\epsilon}(\tau)~d\tau\notag\\
\leq&-\frac{48}{\alpha}\Big((1+e^{1/\alpha})g+2\Big(1+\frac{C_2}{C_1}g^{1/\alpha}\Big)\Big)\notag \\
&\quad\times\frac{1}{\sigma-1-8\epsilon}\frac{1}{(1-\epsilon)^{\sigma+1-8\epsilon}}\frac{1}{m(0)}
C_2g^{1/\alpha}(q^{1-\sigma+8\epsilon}(t)-(1-\epsilon)^{\sigma-1-8\epsilon}) \notag\\
<&\epsilon C_2g^{1/\alpha}q^{1-\sigma+8\epsilon}(t) \label{I:I2}
\end{align}
for any $t\in[0,T_4]$. The first inequality uses \eqref{I:K2}, \eqref{I:X1} and \eqref{I:K1}, \eqref{I:X2},
and the second inequality uses Lemma~\ref{lem:q} and \eqref{I:sigma}. Here one may assume without loss of generality that $\|u_0'\|_{L^\infty}=-m(0)$; we take $-u$ otherwise. The third inequality use \eqref{I:s>1}, and the last inequality uses that \eqref{A3:m2} and \eqref{I:sigma} imply 
\[
-\epsilon(1-\epsilon)^{\sigma+1-8\epsilon}m(0)>\frac{48}{\alpha}\frac{1}{\sigma-1-8\epsilon}
\Big((1+e^{1/\alpha})g+2\Big(1+\frac{C_2}{C_1}g^{1/\alpha}\Big)\Big)
\]
for $\epsilon>0$ sufficiently small. Indeed, $\sigma+1-8\epsilon=5/2-2\epsilon$ and $\sigma-1-8\epsilon=1/2-2\epsilon$ by \eqref{I:sigma}, $m(0)<-1$ by hypotheses, $C_2/C_1<1/2$ by \eqref{def:C}, and replace $\epsilon$ by $\epsilon/8$. 
Evaluating \eqref{e:v2I} at $t=T_4$ and $x=x_2$, we then show that 
\begin{align*}
|(\partial_x^2X)&(T_4;x_2)|\\
=&|v_1^{-1}(T_4;x_2)||u_0''(x_2)-I_2(T_4;x_2)-v_2(T_4;x_2)(\partial_xX)(T_4;x_2)^2| \\
<&-\frac32\frac{1}{m(0)}q(T_4)(g^{1/\alpha}+\epsilon C_2g^{1/\alpha}q^{1-\sigma+8\epsilon}(T_4)
+4C_2g^{1/\alpha}q^{-1-\sigma}(T_4)q^{2-2\epsilon}(T_4)) \\
\leq&-\frac32(5+\epsilon)\frac{1}{m(0)}C_2g^{1/\alpha}q^{2-\sigma-2\epsilon}(T_4)\\
<&-\frac{8}{m(0)}C_2g^{1/\alpha}q^{2-\sigma-2\epsilon}(T_4).
\end{align*}
Therefore, \eqref{claim:X2} holds throughout the interval $[0,T_2]$. Here the first inequality uses \eqref{I:2/3}, \eqref{def:q} and \eqref{I:Gevrey}, \eqref{I:I2}, \eqref{I:vn}, \eqref{I:X1}, the second inequality uses \eqref{A3:m1}, \eqref{def:C} and Lemma~\ref{lem:q}, \eqref{I:sigma}, and the last inequality follows for $\epsilon>0$ sufficiently small. 

\

Similarly, we use \eqref{def:I3} to compute that
\begin{align}
|I_3(t;x_2)|<&\int^t_0\Big(8\frac{6}{\alpha}(1+e^{1/\alpha})3gC_2(2g)^{2/\alpha}
q^{-1-\sigma\alpha-2\sigma}(\tau)q^{3-3\epsilon}(\tau) \notag\\ 
&\qquad -48\frac{6}{\alpha}(1+e^{1/\alpha})2gC_2^2g^{2/\alpha}\frac{1}{m(0)}
q^{-1-\sigma\alpha-\sigma}(\tau)q^{1-\epsilon}(\tau)q^{2-\sigma-2\epsilon}(\tau) \notag \\
&\qquad+\frac{6}{\alpha}(C_1+C_2g^{1/\alpha})\frac{\epsilon}{m^2(0)}C_2^2(2g)^{2/\alpha}
q^{-1-\sigma\alpha}(\tau)q^{3-2\sigma+7\epsilon}(\tau)\Big)~d\tau \notag\\
\leq&\frac{12}{\alpha}\Big(12(1+e^{1/\alpha})g\Big(\frac{1}{C_2}-\frac{2^{2-2/\alpha}}{m(0)}\Big)
-\Big(1+\frac{C_2}{C_1}g^{1/\alpha}\Big)\frac{\epsilon}{m(0)}\Big)  \notag \\
&\hspace*{180pt}\times C_2^2(2g)^{2/\alpha}\int^t_0q^{1-2\sigma+7\epsilon}(\tau)~d\tau \notag\\
\leq&\frac{12}{\alpha}\Big(12(1+e^{1/\alpha})g\Big(\frac{1}{C_2}-\frac{2^{2-2/\alpha}}{m(0)}\Big)
-\Big(1+\frac{C_2}{C_1}g^{1/\alpha}\Big)\frac{\epsilon}{m(0)}\Big) \notag \\
&\times \frac{1}{2\sigma-2-7\epsilon}\frac{1}{(1-\epsilon)^{2\sigma-7\epsilon}}\frac{1}{m(0)}
C_2^2(2g)^{2/\alpha}(q^{2-2\sigma+7\epsilon}(t)-(1-\epsilon)^{2\sigma-2-7\epsilon}) \notag \\
<&-\frac{\epsilon^2}{m(0)}C_2^2(2g)^{2/\alpha}q^{2-2\sigma+7\epsilon}(t) \label{I:I3}
\end{align}
for any $t\in[0,T_4]$. 
The first inequality uses \eqref{I:K2}, \eqref{I:X1}, \eqref{I:X2} and \eqref{I:K1}, \eqref{I:X3},
and the second inequality uses that \eqref{I:sigma} implies that 
$2-\sigma\alpha-2\sigma-3\epsilon>1-2\sigma+7\epsilon$.
Here one may assume without loss of generality that $\|u_0'\|_{L^\infty}=-m(0)$. 
The third inequality uses \eqref{I:s>1}, and the last inequality uses that \eqref{A3:m2} implies that
\[
\epsilon^2(1-\epsilon)^{2\sigma-7\epsilon}(-m(0))^{3/4}>
\frac{(12)^2}{\alpha}\frac{2}{2\sigma-2-7\epsilon}(1+2^{2-2/\alpha})(1+e^{1/\alpha})
\]
and
\[
-\epsilon(1-\epsilon)^{2\sigma-7\epsilon}m(0)>
\frac{24}{\alpha}\frac{1}{2\sigma-2-7\epsilon}\Big(1+\frac{C_2}{C_1}g^{1/\alpha}\Big)
\]
for $\epsilon>0$ sufficiently small. Indeed, $2\sigma-7\epsilon=3+5\epsilon$ and $2\sigma-2-7\epsilon=1+5\epsilon$ by \eqref{I:sigma}, $m(0)<-1$ by hypotheses, $C_2/C_1<1/2$ by \eqref{def:C}, and replace $\epsilon$ by $\epsilon/12$. Integrating \eqref{e:X3'} over the interval $[0,T_4]$, we then show that 
\begin{align*}
|(\partial_x^3X)(T_4;x_2)|\leq&\int^{T_4}_0(|u_0'''(x_2)|+|I_3(t;x_2)|)~dt \\
<&\int^{T_4}_0\Big((2g)^{2/\alpha}
-\frac{\epsilon^2}{m(0)}C_2^2(2g)^{2/\alpha}q^{2-2\sigma+7\epsilon}(t)\Big)~dt \\
\leq&\Big(\frac{1}{C_2^2}-\frac{\epsilon^2}{m(0)}\Big)
\frac{1}{2\sigma-3-7\epsilon}\frac{1}{(1-\epsilon)^{2\sigma-1-7\epsilon}}\frac{1}{m(0)} \\
&\hspace*{80pt}\times C_2^2(2g)^{2/\alpha}
(q^{3-2\sigma+7\epsilon}(T_4)-(1-\epsilon)^{2\sigma-3-7\epsilon})\\
<&\frac{\epsilon}{m^2(0)}C_2^2(2g)^{2/\alpha}q^{3-2\sigma+7\epsilon}(T_4).
\end{align*}
Therefore, \eqref{claim:X3} holds throughout the interval $[0,T_2]$. 
Here the second inequality uses \eqref{I:Gevrey} and \eqref{I:I3}, the third inequality uses \eqref{I:s>1}, and the last inequality uses that \eqref{A3:m2} implies that
\[
\epsilon^2(1-\epsilon)^{2\sigma-1-7\epsilon}(-m(0))^{1/2}>\frac25
\]
for $\epsilon>0$ sufficiently small, satisfying $(1-\epsilon)^{2\sigma-1-7\epsilon}>2/5$. Indeed, $2\sigma-1-7\epsilon=2+5\epsilon$ and $2\sigma-3-7\epsilon=5\epsilon$ by \eqref{I:sigma}, $m(0)<-1$ by hypotheses, and recall \eqref{def:C}.

\

To proceed, since $v_2(T_2;x_2)={\displaystyle \max_{x\in\mathbb{R}}|v_2(T_2;x)|}$, it follows that \[v_3(T_2;x_2)(\partial_xX)(T_2;x_2)=0.\] We multiply \eqref{e:v2I} by $3v_2(\partial_xX)$ and \eqref{e:v3I} by $v_1$ and we take their difference to show that 
\begin{align*}
v_2^2(T_2;x_2)=&\frac13(\partial_xX)^{-3}(T_2;x_2)
(v_1^2(T_2;x_2)(\partial_x^3X)(T_2;x_2) \\
&\hspace*{80pt}+3v_2(T_2;x_2)(\partial_xX)(T_2;x_2)(u_0''(x_2)-I_2(T_2;x_2)) \\
&\hspace*{80pt}-v_1(T_2;x_2)(u_0'''(x_2)-I_3(T_2;x_2))) \\
<&\frac83q^{-3-6\epsilon}(T_2)
\Big(m^2(0)\frac{\epsilon}{m^2(0)}C_2^2(2g)^{2/\alpha}q^{-2}(T_2)q^{3-2\sigma+7\epsilon}(T_2) \\
&\hspace*{60pt}+6C_2g^{1/\alpha}q^{-1-\sigma}(T_2)q^{1-\epsilon}(T_2)
(g^{1/\alpha}+\epsilon C_2g^{1/\alpha}q^{1-\sigma+8\epsilon}(T_2)) \\
&\hspace*{60pt}-m(0)q^{-1}(T_2)
\Big((2g)^{2/\alpha}-\frac{\epsilon^2}{m(0)}C_2^2(2g)^{2/\alpha}q^{2-2\sigma+7\epsilon}(T_2)\Big)\Big) \\
<&\frac83\Big(\epsilon+6\cdot2^{-2/\alpha}\Big(\frac{1}{C_2}+\epsilon\Big)
+\Big(\frac{1}{(-m(0))^{1/2}}+\epsilon^2\Big)\Big)C_2^2(2g)^{2/\alpha}q^{-2-2\sigma+\epsilon}(T_2)\\
<&C_2^2g^{2/\alpha}q^{-2-2\sigma}(T_2).
\end{align*}
Therefore, \eqref{claim:vn} holds for $n=2$ throughout the interval $[0,T_2]$. 
Here the first inequality uses \eqref{I:X1}, \eqref{I:2/3}, \eqref{def:q}, \eqref{I:X3},\eqref{I:vn}
and \eqref{I:Gevrey}, \eqref{I:I2}, \eqref{I:I3},
the second inequality uses \eqref{A3:m1}, \eqref{def:C} and Lemma~\ref{lem:q}, \eqref{I:sigma},
and the last inequality uses that
\[
\epsilon+6\cdot 2^{-2/\alpha}(\epsilon^{3/4}+\epsilon)+\epsilon^{1/2}+\epsilon^2<3\cdot 2^{-3-2/\alpha}
\]
for $\epsilon>0$ sufficiently small.

\

To summarize, a contradiction proves that \eqref{claim:v0}, \eqref{claim:v1} and \eqref{claim:vn} hold
for any $n=0,1,2,\dots$ throughout the interval $[0,T_1]$. 

\

To proceed, note that
\begin{align*}
|\phi_1(t;x)|\leq &\frac{6}{\alpha}(C_1+C_2g^{1/\alpha})q^{-1-\sigma\alpha}(t) \\
<&\frac{6}{\alpha}(C_1+C_2g^{1/\alpha})q^{-2}(t) \\
<&\frac{6}{\alpha}(C_1+C_2g^{1/\alpha})m^{-2}(0)m^2(t)\\ <&\epsilon^2m^2(t)
\end{align*}
for any $t\in [0,T_1]$ for any $x\in\mathbb{R}$. The first inequality uses \eqref{I:K1}, the second inequality uses Lemma~\ref{lem:q} and \eqref{I:sigma}, and the third inequality uses \eqref{def:q}. One may assume without loss of generality that $\|u_0'\|_{L^\infty(\mathbb{R})}=-m(0)$. The last inequality uses that \eqref{A3:m2} implies that 
\[
-\epsilon^2m(0)>\frac{12}{\alpha}\Big(1+\frac{C_2}{C_1}g^{1/\alpha}\Big)
\]
for $\epsilon>0$ sufficiently small. Indeed, $m(0)<-1$ by hypotheses, and $C_2/C_1<1/2$ by \eqref{def:C}. A contradiction therefore proves \eqref{claim:A}. Furthermore, \eqref{claim:v0}, \eqref{claim:v1}, \eqref{claim:vn} hold for any $n=0,1,2,\dots$ 
throughout the interval $[0, T']$ for any $T'<T$.

\

To conclude, let $x \in \Sigma_\epsilon(t)$ for $t\in [0,T)$. 
It follows from \eqref{def:r} and \eqref{I:dr/dt} that 
\[
m(0)(v_1^{-1}(0;x)+(1+\epsilon) t) \leq r(t;x)\leq m(0)(v_1^{-1}(0;x)+(1-\epsilon) t).
\]
Moreover, it follows from Lemma~\ref{lem:S} that $m(0)<v_1(0;x)\leq (1-\epsilon)m(0)$. Hence,
\[
1+m(0)(1+\epsilon)t \leq r(t) \leq \frac{1}{1-\epsilon} + m(0) (1-\epsilon)t.
\]
Furthermore, it follows from \eqref{I:qr} that 
\[
(1-\epsilon) + m(0)(1-\epsilon^2)t \leq q(t) \leq \frac{1}{1-\epsilon} + m(0) (1-\epsilon)t.
\]
Since the function on the left side decreases to zero as ${\displaystyle t\to -\frac{1}{m(0)}\frac{1}{1+\epsilon}}$ and
since the function on the right side decreases to zero as ${\displaystyle t\to -\frac{1}{m(0)}\frac{1}{(1-\epsilon)^2}}$,
therefore, $q(t)\to 0$ and, hence (see \eqref{def:q}), $m(t)\to-\infty$ as $t\to T-$, 
where $T$ satisfies \eqref{E:T}.
On the other hand, \eqref{claim:v0} dictates that $v_0(t;x)$ remains bounded 
for any $t\in [0,T']$, $T'<T$, for any $x\in \mathbb{R}$.
In other words, $\inf_{x\in\mathbb{R}}\partial_xu(x,t)\to -\infty$ as $t\to T-$
but $u(x,t)$ is bounded for any $x\in\mathbb{R}$ for any $t\in [0,T)$. This completes the proof.

\section{Proof of Theorem \ref{thm:whitham}}\label{sec:whitham}

Throughout the section, we take $a=b=1$.

We assume that the initial value problem associated with \eqref{E:whitham}-\eqref{def:K}, where $c=c_{\rm ww}$ (see \eqref{def:cWW}), and $u(\cdot,0)=u_0$  possesses a unique solution in $C^\infty([0,T); H^\infty(\mathbb{R}))$ for some $T>0$. As a matter of fact, one may work out the local-in-time well-posedness in $H^{s}(\mathbb{R})$ for $s>3/2$. Without recourse to the dispersion effects, the proof is identical to that for \eqref{E:fKdV}. Hence, we omit the details. We assume that $T$ is the maximal time of existence. 

Note that $K(x)$ is even and vanishes as $|x|\to\infty$ faster than any polynomial. Since its Fourier transform $c_{\rm ww}(\xi)=\sqrt{\tanh\xi/\xi}$ behaves like $|\xi|^{-1/2}$ as $|\xi|\to \infty$, Whitham~\cite{Whitham} formally argued that $K(x)$ would behave like $|x|^{-1/2}$ as $|x|\to 0$. Recently, Ehrnstr\"om and Wahl\'en~\cite{EW16} analytically confirm it. 

\begin{lemma}\label{lem:K}It follows that 
\[
K(x) \sim \frac{1}{\sqrt{2\pi|x|}} \quad\text{and}\quad 
K'(x)\sim -\frac12\frac{{\rm sgn}(x)}{\sqrt{2\pi|x|^3}}\qquad 
\text{as }|x|\to 0.
\]
\end{lemma}

We include the proof in \cite{EW16} in Appendix~\ref{sec:K} for completeness.

Therefore,
\begin{equation}\label{e:k0}
K(x)\leq \frac{K_0}{\sqrt{|x|}}\quad \text{and}\quad K'(x)\leq \frac{K_0}{\sqrt{|x|^3}}
\qquad \text{for $0<|x|<\delta_0$}
\end{equation}
for some $K_0>0$ a constant and
\begin{equation}\label{e:k-large}
\int^\infty_{\delta_0} |K'(x)|~dx\leq K_\infty
\end{equation}
for some $K_\infty>0$ for some $0<\delta_0<1$.

Recall the notation of the Section~\ref{sec:2/3}, where
\[
\phi_n(t;x)=\int^\infty_{-\infty}K(y)(\partial_x^{n+1}u)(X(t;x)-y,t)~dy
\]
instead of \eqref{def:Kn}. Since $u(x,t)$ is smooth and square integrable in $x$ and smooth in $t$ for any $x\in\mathbb{R}$ for any $t\in[0,T)$, and $X(t;x)$ is continuously differentiable in $t$ and smooth in $x$ for any $x\in\mathbb{R}$ for any $t\in[0,T)$, and since $K$ is square integrable, it follows that $\phi_n(t,x)$ is continuously differentiable in $t$ and smooth and uniformly bounded in $x$ for any $t\in[0,T)$ for any $x\in\mathbb{R}$.

For $0<\delta<\delta_0$, we split the integral and perform an integration by parts to show that
\begin{align}\label{eW:Kn}
|\phi_n(t;x)|=&\Big|\Big(\int_{|y|<\delta}+\int_{|y|>\delta}\Big)
K(y)(\partial_x^{n+1}u)(X(t;x)-y,t))~dy\Big|\notag \\
\leq&\Big| \int_{|y|<\delta}K(y)(\partial_x^{n+1}u)(X(t;x)-y,t)~dy\Big| \notag \\
&+\Big|K(\delta)(\partial_x^nu)(X(t;x)-\delta,t)-K(-\delta)(\partial_x^nu)(X(t;x)+\delta,t)\Big| \notag \\
&+\Big|\int_{|y|>\delta}K'(y)(\partial_x^nu)(X(t;x)-y,t)~dy\Big|\notag \\
\leq&\Big(\int_{|y|<\delta}\frac{K_0}{\sqrt{|y|}}~dy\Big)\|v_{n+1}(t;\cdot)\|_{L^\infty(\mathbb{R})} 
+2|K(\delta)|\|v_n(t;\cdot)\|_{L^\infty(\mathbb{R})} \notag \\
&+\Big(\int_{\delta<|y|<\delta_0}\frac{K_0}{\sqrt{|y|^3}}~dy
+\int_{|y|>\delta_0}|K'(y)|~dy\Big)\|v_n(t;\cdot)\|_{L^\infty(\mathbb{R})}\notag \\
\leq&4K_0\delta^{1/2}\|v_{n+1}(t;\cdot)\|_{L^\infty(\mathbb{R})}
+2K_0\delta^{-1/2}\|v_n(t;\cdot)\|_{L^\infty(\mathbb{R})}\notag \\
&+(4K_0(\delta^{-1/2}-\delta_0^{-1/2})+2K_\infty)\|v_n(t;\cdot)\|_{L^\infty(\mathbb{R})} \notag\\
\leq &C(\delta^{-1/2}\|v_n(t;\cdot)\|_{L^\infty(\mathbb{R})}+\delta^{1/2}\|v_{n+1}(t;\cdot)\|_{L^\infty(\mathbb{R})})
\end{align}
for some $C>0$ for $n=0,1,2,\dots$ and for any $t\in [0,T)$ for any $x\in\mathbb{R}$.
The first inequality uses \eqref{e:k0}, the second inequality uses \eqref{e:k-large},
and the last inequality uses that $0<\delta_0<1$. After submitting the manuscript, the author has learned that \eqref{eW:Kn} follows from the Gagliardo-Nirenberg interpolation inequality for a broad class of $K$; see \cite{Hur2017} for details. We merely pause to remark that $K$ is integrable near zero, although $K(0)$ does not exist, whence the arguments in \cite{Seliger} and \cite{CEbreaking} do not apply. But $K'$ is not integrable near zero,whence the argument in \cite{NS} does not directly apply.

We are done if 
\begin{equation}\label{claim:Ak}
|\phi_1(t;x)|<\epsilon^2m^2(t) \qquad\text{for any $t\in[0,T)$}\quad\text{for any $x\in\mathbb{R}$}
\end{equation}
for $\epsilon>0$ sufficiently small. It follows from \eqref{AW:m1} and the Sobolev inequality that 
\[
|\phi_1(0;x)|=\Big|\int^\infty_{-\infty}K(x-y)u_0''(y)~dy\Big|\leq\|u_0\|_{H^{2+}(\mathbb{R})}<\epsilon^2m^2(0)
\qquad\text{for any $x\in\mathbb{R}$.}
\]
In other words, \eqref{claim:Ak} holds at $t=0$. 
Suppose on the contrary that $|\phi_1(T_1;x)|=\epsilon^2m^2(T_1)$ 
for some $T_1\in(0,T)$ for some $x\in\mathbb{R}$. By continuity, we may assume that
\[
|\phi_1(t;x)|<\epsilon^2m^2(t) \qquad\text{for any $t\in[0,T_1]$}\quad\text{for any $x\in\mathbb{R}$}.
\] 
Under the assumption, we rerun the argument in the previous section to show that 
Lemma~\ref{lem:S}, Lemma~\ref{lem:q}, Lemma~\ref{lem:qs} hold.

We claim \eqref{claim:v0}, \eqref{claim:v1} and \eqref{claim:vn} hold, where 
$C_0$, $C_1$, $C_2$ are in \eqref{def:C}, $\sigma$ is in \eqref{I:sigma} but $\alpha=1/2$.
It follows from \eqref{def:C}, \eqref{def:q} and \eqref{IW:Gevrey} that 
\eqref{claim:v0}, \eqref{claim:v1} and \eqref{claim:vn} hold for any $n=0,1,2,\dots$ at $t=0$. 
Suppose on the contrary that \eqref{claim:v0}, \eqref{claim:v1} and \eqref{claim:vn} hold for any $n=0,1,2,\dots$
throughout the interval $[0,T_2)$ but do not for some $n\geq 0$ at $t=T_2$ for some $T_2\in(0,T_1]$.
By continuity, \eqref{I:v0}, \eqref{I:v1} and \eqref{I:vn} hold for any $n=0,1,2,\dots$ for any $t\in[0,T_2]$.

For $n=0$, it follows from \eqref{eW:Kn}, where $\delta(t)=\delta_0q(t)$, and \eqref{I:v0}, \eqref{I:v1} that 
\begin{align}
|\phi_0(t;x)|\leq C(C_0\delta_0^{-1/2}q^{-1/2}(t)+C_1\delta_0^{1/2}q^{1/2}(t)q^{-1}(t)
<C\delta_0^{-1/2}(C_0+C_1)q^{-1}(t)\label{IW:K0}
\end{align}
for any $t\in[0,T_2]$ for any $x\in \mathbb{R}$. The last inequality uses that $0<\delta_0<1$.
For $n=1$, similarly, it follows from \eqref{eW:Kn}, where $\delta(t)=\delta_0q^\sigma(t)$, 
and \eqref{I:v1}, \eqref{I:vn}, where $\alpha=1/2$, that
\begin{align}
|\phi_1(t;x)|\leq &C(C_1\delta_0^{-1/2}q^{-\sigma/2}(t)q^{-1}(t)
+C_2g^2\delta_0^{1/2}q^{\sigma/2}(t)q^{-1-\sigma}(t)) \notag \\
<&C\delta_0^{-1/2}(C_1+C_2g^2)q^{-1-\sigma/2}(t) \label{IW:K1}
\end{align}
for any $t\in[0,T_2]$ for any $x\in\mathbb{R}$. The last inequality, similarly, uses that $0<\delta_0<1$. 
For $n\geq 2$, moreover, it follows from \eqref{eW:Kn}, where $\delta(t)=\delta_0(ng)^{-2}q^\sigma(t)$, 
and \eqref{I:vn}, where $\alpha=1/2$, that 
\begin{align}
|\phi_n(t;x)|\leq &C(\delta_0^{-1/2}C_2((n-1)g)^{2(n-1)}q^{-\sigma/2}(t)q^{-1-(n-1)\sigma}(t) \notag \\
&\hspace*{100pt}+\delta_0^{1/2}C_2(ng)^{2n}q^{\sigma/2}(t)q^{-1-n\sigma}(t)) \notag \\
<&C\delta_0^{-1/2}(1+e^2)(ng)C_2((n-1)g)^{2(n-1)}q^{-1-\sigma/2-(n-1)\sigma}(t) \label{IW:Kn}
\end{align}
for any $t\in[0,T_2]$ for any $x\in\mathbb{R}$. The last inequality, similarly, uses that $0<\delta_0<1$. 

We may rerun the argument in the previous section but we use \eqref{IW:K0}, \eqref{IW:K1}, \eqref{IW:Kn} instead of \eqref{I:K0}, \eqref{I:K1}, \eqref{I:K2}, respectively, use \eqref{AW:m1} and \eqref{IW:Gevrey} instead of \eqref{A3:m1}-\eqref{A3:m3} and \eqref{I:Gevrey}, respectively, and necessarily choose $\epsilon>0$ smaller in various places, to draw a contradiction. The details are nearly identical to those in the previous section. Hence, we omit the details.
To conclude, \eqref{claim:v0}, \eqref{claim:v1} and \eqref{claim:vn} hold for any $n=0,1,2,\dots$ throughout the interval $[0,T_1]$, where $\alpha=1/2$. 

To proceed, we necessarily choose $\epsilon$ smaller, and it follows from \eqref{IW:K1}, Lemma~\ref{lem:q}, \eqref{I:sigma}, \eqref{def:q}, and \eqref{AW:m1} that
\[
|\phi_1(t;x)|\leq C\delta_0^{-1/2}(C_1+C_2g^2)m^{-2}(0)m^2(t)<\epsilon^2m^2(t)
\]
for any $t\in[0,T_1]$ for any $x\in\mathbb{R}$. A contradiction therefore proves \eqref{claim:Ak}.

The remainder of the proof is nearly identical to that in the previous section. Hence, we omit the details. 

\subsection*{Acknowledgment}
The author thanks Mats Ehrnstr\"om for valuable discussions, and anonymous referees for helpful comments and suggestions.
She is supported by the National Science Foundation under the Faculty Early Career Development (CAREER) Award DMS-1352597, an Alfred P. Sloan Research Fellowship, and by the University of Illinois at Urbana-Champaign, the Center for Advanced Study under a Beckman Fellowship. She is grateful to the Mathematisches Forschungsinstitut Oberwolfach for its generous hospitality during the workshop ``Mathematical Theory of Water Waves," where part of the research was carried out.

\begin{appendix}

\section{Assorted proofs of lemmas}\label{app:proofs}

\begin{proof}[Proof of Lemma~\ref{lem:S}]
Suppose on the contrary that $x_1\notin\Sigma_{\gamma}(t_1)$ but $x_1\in\Sigma_{\gamma}(t_2)$
for some $x_1\in\mathbb{R}$ for some $0\leq t_1\leq t_2\leq T_1$. That is,
\begin{equation}\label{def:x1}
v_1(t_1;x_1)>(1-\gamma)m(t_1)\quad\text{and}\quad
v_1(t_2;x_1)\leq(1-\gamma)m(t_2)<\frac12m(t_2).
\end{equation}
We may choose $t_1$ and $t_2$ close so that 
\[
v_1(t;x_1)\leq\frac12m(t)\qquad\text{for any $t\in[t_1,t_2]$.}
\]
Indeed, $v_1(\cdot\,;x_1)$ and $m$ are uniformly continuous throughout the interval $[0,T_1]$. Let 
\begin{equation}\label{def:x2}
v_1(t_1;x_2)=m(t_1)<\frac12m(t_1).
\end{equation}
We may necessarily choose $t_2$ closer to $t_1$ so that 
\[
v_1(t;x_2)\leq\frac12m(t)\qquad \text{for any $t\in[t_1,t_2]$.}
\]
For $\epsilon>0$ sufficiently small, it follows from \eqref{I:A} that 
\[
|\phi_1(t;x_j)|\leq \epsilon^2m^2(t)\leq 4\epsilon^2v_1^2(t;x_j)<\frac\gamma2 v_1^2(t;x_j)
\qquad \text{for any $t\in[t_1,t_2]$ and $j=1,2$.}
\]
To proceed, note from \eqref{e:v1} that 
\[
\frac{dv_1}{dt}(\cdot\,;x_1)=-v_1^2(\cdot\,;x_1)-\phi_1(\cdot\,;x_1)
\geq\Big(-1-\frac{\gamma}{2}\Big)v_1^2(\cdot\,;x_1)
\]
and
\[
\frac{dv_1}{dt}(\cdot\,;x_2)\leq\Big(-1+\frac{\gamma}{2}\Big)v_1^2(\cdot\,;x_2)
\]
throughout the interval $(t_1,t_2)$. Integrating them over the interval $[t_1,t_2]$, we arrive at that
\[
\hspace*{-15pt}v_1(t_2;x_1)\geq\frac{v_1(t_1;x_1)}{1+(1+\frac{\gamma}{2})v_1(t_1;x_1)(t_2-t_1)}
\quad\text{and}\quad
v_1(t_2;x_2)\leq\frac{v_1(t_1;x_2)}{1+(1-\frac{\gamma}{2})v_1(t_1;x_2)(t_2-t_1)}.
\]
The latter inequality and \eqref{def:x2} imply that 
\[
m(t_2)\leq\frac{m(t_1)}{1+(1-\frac{\gamma}{2})m(t_1)(t_2-t_1)}.
\]
The former inequality and \eqref{def:x1}, on the other hand, imply that
\begin{align*}
v_1(t_2;x_1)>&\frac{(1-\gamma)m(t_1)}{1+(1+\frac{\gamma}{2})(1-\gamma)m(t_1)(t_2-t_1)} \\
>&\frac{(1-\gamma)m(t_1)}{1+(1-\frac{\gamma}{2})m(t_1)(t_2-t_1)} \\
\geq& (1-\gamma)m(t_2).
\end{align*}
A contradiction therefore completes the proof.
\end{proof}

\begin{proof}[Proof of Lemma~\ref{lem:n3}]
We use Stirling's inequality to compute that
\begin{align*}
\sum_{j=2}^{n-1}\Big(\begin{matrix}n\\j\end{matrix}\Big)(j&-1)^{(j-1)/\alpha}(n-j)^{(n-j)/\alpha} \\
\leq &\sum_{j=2}^{n-1}\frac{n^n}{j^j(n-j)^{n-j}}(j-1)^{(j-1)/\alpha}(n-j)^{(n-j)/\alpha} \\
=&n\Big(\frac{n}{n-1}\Big)^{n-1}(n-1)^{(n-1)/\alpha}\sum_{j=2}^{n-1}\frac{1}{j}\Big(\frac{j-1}{j}\Big)^{j-1} 
\Big(\frac{(j-1)^{j-1}(n-j)^{n-j}}{(n-1)^{n-1}}\Big)^{1/\alpha-1} \\
\leq&en(n-1)^{(n-1)/\alpha}\sum_{j=2}^{n-1}\frac{1}{j}\Big(\frac{j-1}{n-1}\Big)^{1/\alpha-1} \\
\leq&en(n-1)^{(n-1)/\alpha} \Big(\frac{1}{n-1}\Big)^{1/\alpha-1}\int^n_1y^{1/\alpha-2}~dy\\
\leq&\frac{e}{1/\alpha-1}n(n-1)^{(n-1)/\alpha}\Big(\frac{n}{n-1}\Big)^{1/\alpha-1}.
\end{align*}
The last inequality uses that $0<\alpha<2/3$. 
\end{proof}

\section{Proof of Lemma~\ref{lem:K}}\label{sec:K}

The proof is found in \cite{EW16}, for instance. 

Let's write
\[
K(x)=\frac{1}{2\pi}\lim_{\epsilon\to0}\int^{1/\epsilon}_{-1/\epsilon}
\sqrt{\frac{\tanh\xi}{\xi}}e^{ix\xi}~d\xi.
\]
For $x>0$, we make a straightforward calculation to show that 
\begin{align*}
\sqrt{2\pi x}K(x)=&\sqrt{\frac{x}{2\pi}}\lim_{\epsilon\to0}\int^{1/\epsilon}_{-1/\epsilon}
\sqrt{\frac{\tanh\xi}{\xi}}e^{ix\xi}~d\xi \\
=&\sqrt{\frac{2x}{\pi}}\lim_{\epsilon\to0}\int^{1/\epsilon}_0\sqrt{\frac{\tanh\xi}{\xi}}\cos(x\xi)~d\xi \\
=&-\frac{1}{\sqrt{2\pi x}}\lim_{\epsilon\to0}\int^{1/\epsilon}_0
\frac{1}{\sqrt{\xi}^3}\Big(\frac{2\xi}{\sinh(2\xi)}-1\Big)\sqrt{\tanh\xi}\sin(x\xi)~d\xi\\
=:&\frac{1}{\sqrt{2\pi}}\int^\infty_0\frac{\sin\zeta}{\sqrt{\zeta}^3}f_0(\zeta/x)~d\zeta,
\end{align*}
where
\[
f_0(z)=\Big(1-\frac{2z}{\sinh(2z)}\Big)\sqrt{\tanh z}.
\]
Since $f_0(z)$ is bounded for any $z\in\mathbb{R}$ and $f_0(z)\to 1$ as $z\to\infty$ and 
since it is well known that
\[
\int^\infty_0\frac{\sin z}{\sqrt{z}^3}~dz=\sqrt{2\pi},
\]
it follows from Lebegues' dominated convergence theorem that 
${\displaystyle \lim_{x\to0+}\sqrt{2\pi x}K(x)=1}$. For $x>0$, moreover, 
\[
K(x)=\frac{1}{2\pi}\frac{1}{\sqrt{x}}\int^\infty_0\frac{\sin \zeta}{\sqrt{\zeta}^3}f_0(\zeta/x)~d\zeta
=\frac{1}{2\pi}\int^\infty_0\frac{\sin\zeta}{\zeta^2}\sqrt{\zeta/x}f_0(\zeta/x)~d\zeta.
\]
We then make another straightforward calculation to show that  
\[
\sqrt{x}^3K'(x)=-\frac{1}{2\pi}\int^\infty_0\frac{\sin \zeta}{\sqrt{\zeta}^3}f_1(\zeta/x)~d\zeta,
\]
where $f_1(z)=\sqrt{z}(\sqrt{z}f(z))'$. 
Since $f_1(z)$ is bounded for any $z\in\mathbb{R}$ and $f_1(z)\to1/2$ as $z\to\infty$, 
similarly, it follows from Lebegues' dominated convergence theorem that 
${\displaystyle \lim_{x\to0+}\sqrt{2\pi x^3}K'(x)=-1/2}$.
This completes the proof.

\end{appendix}

\bibliographystyle{amsalpha}
\bibliography{breakingBib}

\newcommand{\etalchar}[1]{$^{#1}$}
\providecommand{\bysame}{\leavevmode\hbox to3em{\hrulefill}\thinspace}
\providecommand{\MR}{\relax\ifhmode\unskip\space\fi MR }
\providecommand{\MRhref}[2]{%
  \href{http://www.ams.org/mathscinet-getitem?mr=#1}{#2}
}
\providecommand{\href}[2]{#2}
\begin{thebibliography}{CKS{\etalchar{+}}03}

\bibitem[BF67]{BF}
T.~B. Benjamin and J.~E. Feir, \emph{The disintegration of wave trains on deep
  water. {P}art 1. {T}heory}, J. Fluid Mech. \textbf{27} (1967), no.~3,
  417--437.

\bibitem[BH10]{BH}
Joseph Biello and John~K. Hunter, \emph{Nonlinear {H}amiltonian waves with
  constant frequency and surface waves on vorticity discontinuities}, Comm.
  Pure Appl. Math. \textbf{63} (2010), no.~3, 303--336. \MR{2599457}

\bibitem[BM95]{BM1995}
Thomas~J. Bridges and Alexander Mielke, \emph{A proof of the {B}enjamin-{F}eir
  instability}, Arch. Rational Mech. Anal. \textbf{133} (1995), no.~2,
  145--198. \MR{1367360 (97c:76028)}

\bibitem[BN14]{BN2014}
Alberto Bressan and Khai~T. Nguyen, \emph{Global existence of weak solutions
  for the {B}urgers-{H}ilbert equation}, SIAM J. Math. Anal. \textbf{46}
  (2014), no.~4, 2884--2904. \MR{3248030}

\bibitem[CCG10]{CCG2010}
Angel Castro, Diego C{\'o}rdoba, and Francisco Gancedo, \emph{Singularity
  formations for a surface wave model}, Nonlinearity \textbf{23} (2010),
  no.~11, 2835--2847. \MR{2727172}

\bibitem[CE98]{CEbreaking}
Adrian Constantin and Joachim Escher, \emph{Wave breaking for nonlinear
  nonlocal shallow water equations}, Acta Math. \textbf{181} (1998), no.~2,
  229--243. \MR{1668586 (2000b:35206)}

\bibitem[CKS{\etalchar{+}}03]{CCKTT2003}
J.~Colliander, M.~Keel, G.~Staffilani, H.~Takaoka, and T.~Tao, \emph{Sharp
  global well-posedness for {K}d{V} and modified {K}d{V} on {$\Bbb R$} and
  {$\Bbb T$}}, J. Amer. Math. Soc. \textbf{16} (2003), no.~3, 705--749
  (electronic). \MR{1969209}

\bibitem[EW16]{EW16}
Mats Ehrnstr\"om and Erik Wahl\'en, \emph{On {W}hitham's conjecture of a
  highest cusped wave for a nonlocal dispersive equation}, arxiv:1602.05384
  (2016).

\bibitem[HJ15]{HJ2}
Vera~Mikyoung Hur and Mathew~A. Johnson, \emph{Modulational instability in the
  {W}hitham equation for water waves}, Stud. Appl. Math. \textbf{134} (2015),
  no.~1, 120--143. \MR{3298879}

\bibitem[H{\"o}r83]{Hormander}
Lars H{\"o}rmander, \emph{The analysis of linear partial differential
  operators. {I}}, Grundlehren der Mathematischen Wissenschaften [Fundamental
  Principles of Mathematical Sciences], vol. 256, Springer-Verlag, Berlin,
  1983, Distribution theory and Fourier analysis. \MR{717035 (85g:35002a)}

\bibitem[HT14]{HT1}
Vera~Mikyoung Hur and Lizheng Tao, \emph{Wave breaking for the {W}hitham
  equation with fractional dispersion}, Nonlinearity \textbf{27} (2014),
  no.~12, 2937--2949. \MR{3291137}

\bibitem[Hur12]{Hur-blowup}
Vera~Mikyoung Hur, \emph{On the formation of singularities for surface water
  waves}, Commun. Pure Appl. Anal. \textbf{11} (2012), no.~4, 1465--1474.
  \MR{2900797}

\bibitem[Hur17]{Hur2017}
\bysame, \emph{Shallow water models with constant vorticity}, Proc. R. Soc. A
  (2017), to appear, arxiv:1701.08817.

\bibitem[Kat83]{Kato}
Tosio Kato, \emph{On the {C}auchy problem for the (generalized) {K}orteweg-de
  {V}ries equation}, Studies in applied mathematics, Adv. Math. Suppl. Stud.,
  vol.~8, Academic Press, New York, 1983, pp.~93--128. \MR{759907 (86f:35160)}

\bibitem[KS15]{KS}
Christian Klein and Jean-Claude Saut, \emph{A numerical approach to blow-up
  issues for dispersive perturbations of {B}urgers' equation}, Phys. D
  \textbf{295/296} (2015), 46--65. \MR{3317254}

\bibitem[Lan13]{Lannes}
David Lannes, \emph{The water waves problem: Mathematical analysis and
  asymptotics}, Mathematical Surveys and Monographs, vol. 188, American
  Mathematical Society, Providence, RI, 2013.

\bibitem[NS94]{NS}
Pavel~I. Naumkin and Ilia~A. Shishmar{\"e}v, \emph{Nonlinear nonlocal equations
  in the theory of waves}, Translations of Mathematical Monographs, vol. 133,
  American Mathematical Society, Providence, RI, 1994, Translated from the
  Russian manuscript by Boris Gommerstadt. \MR{1261868 (94m:35230)}

\bibitem[Sel68]{Seliger}
Robert~Lewis Seliger, \emph{A note on the breaking of waves}, Proc. R. Soc.
  Lond. Ser. A Math. Phys. Eng. Sci. \textbf{303} (1968), no.~1475, 493--496.

\bibitem[Tol78]{Toland1978}
J.~F. Toland, \emph{On the existence of a wave of greatest height and
  {S}tokes's conjecture}, Proc. Roy. Soc. London Ser. A \textbf{363} (1978),
  no.~1715, 469--485. \MR{513927}

\bibitem[Whi67a]{Whitham-BF}
G.~B. Whitham, \emph{Non-linear dispersion of water waves}, J. Fluid Mech.
  \textbf{27} (1967), 399--412. \MR{0208903 (34 \#8711)}

\bibitem[Whi67b]{Whitham1967}
Gerald~B. Whitham, \emph{Variational methods and applications to water waves},
  Proc. R. Soc. Lond. Ser. A Math. Phys. Eng. Sci. \textbf{299} (1967),
  no.~1456, 6--25, A Discussion on Nonlinear Theory of Wave Propagation in
  Dispersive Systems (Jun. 13, 1967).

\bibitem[Whi74]{Whitham}
\bysame, \emph{Linear and nonlinear waves}, Wiley-Interscience [John Wiley \&
  Sons], New York-London-Sydney, 1974, Pure and Applied Mathematics.
  \MR{0483954 (58 \#3905)}

\end{thebibliography}

\end{document}